\documentclass[12pt,reqno]{amsart}

\usepackage{a4wide}
\usepackage{hyperref}
\usepackage{needspace}
\usepackage[inline]{enumitem}
\usepackage{pgfplots}
\pgfplotsset{compat=1.18}

\usepackage{tikz}
\usetikzlibrary{decorations.pathreplacing, tikzmark}
\usetikzlibrary{shapes.geometric}
\usetikzlibrary{positioning}

\newtheorem{theorem}[subsection]{Theorem}
\newtheorem{lemma}[subsection]{Lemma}

\newtheorem{proposition}[subsection]{Proposition}
\newtheorem*{proposition*}{Proposition}
\newtheorem*{lemma*}{Lemma}
\newtheorem*{claim*}{Claim}

\providecommand{\Z}{\mathbb{Z}}
\providecommand{\N}{\mathbb{N}}
\providecommand{\C}{\mathbb{C}}
\providecommand{\R}{\mathbb{R}}

\providecommand{\T}{\mathbb{T}}

\providecommand{\wh}{\widehat}
\providecommand{\wt}{\widetilde}
\newcommand{\dd}{\,\mathrm{d}}

\providecommand{\supp}{\mathop{\rm supp}\nolimits}

\providecommand{\Bohr}{\mathop{\rm Bohr}\nolimits}
\renewcommand{\Re}{\mathop{\rm Re}\nolimits}

\numberwithin{equation}{section}

\newcounter{constcntbig} 
\newcounter{constcntlittle} 

\makeatletter
\newcommand{\newconstbig}[1]{%
    \refstepcounter{constcntbig}%
    \hypertarget{const:#1}{\mbox{}}%
    \protected@write\@auxout{}%
        {\string\newlabel{const:#1}{{\arabic{constcntbig}}{\thepage}}}%
    C_{\arabic{constcntbig}}%
}
\makeatother
\makeatletter
\newcommand{\newconstlittle}[1]{%
    \refstepcounter{constcntlittle}%
    \hypertarget{const:#1}{\mbox{}}%
    \protected@write\@auxout{}%
        {\string\newlabel{const:#1}{{\arabic{constcntlittle}}{\thepage}}}%
    c_{\arabic{constcntlittle}}%
}
\makeatother

\newcommand{\refconstbig}[1]{%
    \hyperlink{const:#1}{C_{\ref{const:#1}}}%
}

\begin{document}

\title{The quantitative Beurling-Helson theorem}

\author{Tom Sanders}
\address{Institute of Mathematics, University of Oxford, United Kingdom}
\email{tom.sanders@maths.ox.ac.uk}

\begin{abstract}
We show that for any $\varepsilon>0$ if $\phi:\T \rightarrow \T$ is continuous and $\|\exp(-2\pi i z \phi)\|_{A(\T)} =O_{|z|\rightarrow \infty}(\log^{\frac{1}{8}-\varepsilon} |z|)$ then $\phi(x)=wx+t$ for some $w \in\Z$ and $t \in \T$.
\end{abstract}

\maketitle

\section{Introduction}

Following Rudin \cite{rud::1}, for $G$ a compact Abelian group with Haar probability measure $\mu_G$ and dual group $\Gamma$ we define the Fourier transform of $f \in L_1(\mu_G)$ to be
\begin{equation}\label{eqn.ft}
\wh{f}(\gamma):=\int{f(x)\overline{\langle x,\gamma\rangle}\dd \mu_G(x)}.
\end{equation}
The space $A(G)$ is the functions $f \in L_1(\mu_G)$ such that $\wh{f} \in \ell_1(\Gamma)$ and is a Banach algebra with the norm $\|f \|_{A(G)}:=\|\wh{f}\|_{\ell_1(\Gamma)}$.

The Beurling-Helson theorem of the title is \cite[Theorem, p126]{Beurling:1953aa} in the special case of the circle group $\T:=\R/\Z$:
\begin{theorem}
Suppose that $\phi:\T \rightarrow \T$ is continuous and $$\|\exp(-2\pi i z\phi)\|_{A(\T)} = O_{|z| \rightarrow \infty}(1).$$ Then $\phi(x) = wx+t$ for some $w \in \Z$ and $t \in \T$.
\end{theorem}
Kahane \cite[p121]{Kahane:1963aa} conjectures that the weaker condition
\begin{equation*}
\|\exp(-2\pi i z\phi)\|_{A(\T)} = o_{|z|\rightarrow \infty}(\log |z|)
\end{equation*}
would suffice for the same conclusion, and notes that if true this is the best possible result as nonlinear piecewise linear functions $\phi$ have $\|\exp(-2\pi i z\phi)\|_{A(\T)} =O_{|z| \rightarrow \infty}(\log |z|)$. 

This last calculation of Kahane was strengthened by Lebedev in \cite{Lebedev:2012ab}, who also in a separate paper proved the first result towards Kahane's conjecture:
\begin{theorem}[{\cite[Theorem, p122]{Lebedev:2012aa}}]
Suppose that $\phi:\T \rightarrow \T$ is continuous and $$\|\exp(-2\pi i z\phi)\|_{A(\T)} = o_{|z| \rightarrow \infty}\left(\frac{\log \log |z|}{\log \log \log |z|}\right)^{\frac{1}{12}}.$$ Then $\phi(x) = wx+t$ for some $w \in \Z$ and $t \in \T$.
\end{theorem}
Lebedev's argument was refined by Konyagin and Shkredov:
\begin{theorem}[{\cite[Theorem 2, p110]{konshk::0}}]
Suppose that $\phi:\T \rightarrow \T$ is continuous and $$\|\exp(-2\pi i z\phi)\|_{A(\T)} = o_{|z| \rightarrow \infty}\left(\frac{\log |z|}{( \log \log |z|)^6}\right)^{\frac{1}{22}}.$$ Then $\phi(x) = wx+t$ for some $w \in \Z$ and $t \in \T$.
\end{theorem}
Our aim is to show the following:
\begin{theorem}\label{thm.main}
Suppose that $\phi:\T \rightarrow \T$ is continuous and $\varepsilon>0$ is such that $$\|\exp(-2\pi i z\phi)\|_{A(\T)} = O_{|z| \rightarrow \infty}(\log^{\frac{1}{8}-\varepsilon}  |z|).$$ Then $\phi(x) = wx+t$ for some $w \in \Z$ and $t \in \T$.
\end{theorem}
There is considerable further discussion of the problem in \cite{Lebedev:2012aa}.

\subsection*{Notation} We use big-$O$ notation in the usual way in the paper, but also absolute constants $C_1,C_2,\dots> 1$ and $c_1,c_2,\dots \in (0,1)$. All of the absolute constants could be calculated. We give them names to help make clear the dependencies in some proofs, and in particular that those dependencies are not circular.

It will also be useful to generalise our notation for Haar measure on a compact Abelian group to any open $S \subset G$, by writing $\mu_S$ for the measure $\mu_G$ restricted to $S$ and normalised to have total mass $1$.

Finally, we write $\T_N:=\{x/N + \Z: x \in \Z\} \leq \T$ and equip it with Haar probability measure, and then $\Z_N:=\{x+N\Z: x \in \Z\}$ is a dual of $\T_N$ with the usual pairing.

\section{Proof of Theorem \ref{thm.main}}\label{sec.proof}

We pursue the same ideas as Lebedev \cite{Lebedev:2012aa} with the improvements of Konyagin and Shkredov \cite{konshk::0}. Lebedev begins by considering
\begin{equation*}
E:=\{(x,y,z) \in \T^3: \phi(x+y+z)-\phi(x+y)-\phi(x+z)+\phi(x)=0\},
\end{equation*}
which he shows is a large subset of $\T^3$ by an argument \cite[Lemma 2, p126]{Lebedev:2012aa} which he attributes to Kahane. The hypothesis -- that is the bound on $\|\exp(-2\pi i z\phi)\|_{A(\T)}$ as a function of $z$ -- implies that a suitably smoothed version of
\begin{equation*}
\Phi(x,y,z):=\exp(2\pi i (\phi(x+y+z)-\phi(x+y)-\phi(x+z)+\phi(x)))
\end{equation*}
is in $A(\T^3)$.

Functions in $A(\T^3)$ are continuous in a quantitative sense: they have a set of positive measure over which they do not vary very much, and this measure has a lower bound depending only on the norm of the function. This can be used to show that $E$, which is a level set of $\Phi$, is almost the whole of $\T^3$ and then a limiting argument tells us that $E=\T^3$ and this implies $\phi$ has the required form.

The quantitative continuity of functions in $A(\T^3)$ is captured in a result of Green and Konyagin \cite{grekon::} recorded in \cite[(2), p124]{Lebedev:2012aa}, and the final limiting argument is then on \cite[p129]{Lebedev:2012aa}.

The final ingredient in Lebedev's argument is a discrete approximation \cite[p124]{Lebedev:2012aa} which allows him to take advantage of the formulation of Green and Konyagin's results recorded in \cite{grekon::}. This approximation was considerably improved by Konyagin and Shkredov in \cite{konshk::0}.

The main difference of the present paper is that instead of considering the set $E$, we consider the graph $\Gamma:=\{(x,\phi(x)): x \in \T\}$. The downside of this is that $\Gamma$ is not large in the same way that $E$ is, but we can recover the situation by using the standard additive combinatorial tools of Freiman's Theorem and the Balog-Szemer{\'e}di-Gowers theorem. On the other hand, the upside is that the algebra norm of (a discretisation of) $\Gamma$ is smaller than the corresponding norm when considering the set $E$. This saving leads to our improvement.

We turn now to recording the three main parts of our argument before giving the proof. 

\begin{proposition}[Proposition \ref{prop.datks}]
Suppose that $C \geq 1$ and $c \in (0,1/2)$. Then there are constants $F(C,c), H(C,c)\geq 1 $ such that if $N \in \N^*$ and $\phi:\T \rightarrow \T$ is continuous with
\begin{equation*}
\|\exp(-2\pi i z \phi)\|_{A(\T_N)} \leq C\log^c(3+|z|) \text{ for all }z \in \Z.
\end{equation*}
Then there is an integer $1 \leq Q \leq \exp(F(C,c)\log^{\frac{2}{1-2c}}2N)$ and $\phi^*:\T_N \rightarrow \T_Q$ such that $\|\phi(x)-\phi^*(x)\|_\T\leq \frac{H(C,c)}{NQ}$ for all $x \in \T_N$.
\end{proposition}

\begin{proposition}[Proposition \ref{prop.partialcohen}]
Suppose that $G$ is a finite Abelian group, $1_A \in A(G)$ has $\|1_A\|_{A(G)} \leq M$ for some $M\geq 1$, and $\eta \in (0,1/4]$. Then there is $V \leq G$ with $\|1_A \ast \mu_V\|_{L_\infty(G)} \geq 1-\eta$ and $|V| \geq \exp(-\refconstbig{partial}M^{3}\eta^{-3}\log^52M\eta^{-1})|A|$.
\end{proposition}

\begin{lemma}[Lemma \ref{lem.limit}]
Suppose that $\phi:\T \rightarrow \T$ is continuous and for all $\eta \in (0,1]$ and $N_0\in \N^*$ there is an integer $N \geq N_0$, $\phi^*:\T_N \rightarrow \T$ with graph $\Gamma:=\{(x,\phi^*(x)): x \in\T_N\}$, and $W$ such that
\begin{enumerate}[label=(\roman*)]
\item  $\|\phi(x)-\phi^*(x)\|_\T \leq \eta$ for all $x \in \T_N$;
\item  $W$ is a coset of a finite subgroup of $\T^2$ such that $|\Gamma \triangle W| \leq 2\eta N$.
\end{enumerate}
Then there is $t\in \T$ and $w \in \Z$ such that $\phi(x)=wx+t$ for all $x \in \T$.
\end{lemma}

\begin{proof}[Proof of Theorem \ref{thm.main}]
Suppose that $c:=\frac{1}{8}-\varepsilon>0$. From the hypothesis there is $C \geq 1$ such that
\begin{equation}\label{eqn.phibasic2}
\|\exp(-2\pi i z \phi)\|_{A(\T)} \leq C\log^{c} (3+|z|) \text{ for all }z \in \Z.
\end{equation}
Let
\begin{equation*}
F:=F(C,c) \text{ and }H=H(C,c)
\end{equation*}
where these are the constants as in Proposition \ref{prop.datks}.

$\T_N$ is a closed subgroup of $\T$ and so $A(\T_N)$ and $A(\T)$ are well-defined subalgebras of $B(\T_N)$ and $B(\T)$ respectively, and the restriction map $B(\T) \rightarrow B(\T_N); f \mapsto f|_{\T_N}$ is a contractive surjection (this is essentially the proof of \cite[Theorem 2.7.2, p53]{rud::1}). Moreover, the indicator functions of closed subgroups have algebra norm $1$. Combining all these we get that
\begin{align}
\nonumber \|\exp(-2\pi i z \phi)\|_{A(\T_N)}& \nonumber =\|\exp(-2\pi iz\phi)\|_{B(\T_N)}\\ \nonumber &\leq \|\exp(-2\pi i z \phi)1_{\T_N}\|_{B(\T)}\\ \label{eqn.alg} &\leq \|\exp(-2\pi i z \phi)\|_{B(\T)}\|1_{\T_N}\|_{B(\T)}  = \|\exp(-2\pi i z \phi)\|_{A(\T)}
\end{align}
for all $z \in \Z$.

Our aim now is to show that the hypothesis of Lemma \ref{lem.limit} are satisfied. It is enough to establish these for all $\eta$ sufficiently small so let $\eta \in (0,1/4]$. Since there are infinitely many primes we may let $N \geq N_0$ be prime such that
\begin{equation}\label{eqn.vbound}
\frac{H}{N} \leq \eta \text{ and } \exp(-\refconstbig{partial}M^{3}\eta^{-3}\log^52M\eta^{-1})>N^{-1} =\exp(-\log N)
\end{equation}
where
\begin{equation}
Q_0:=\exp(F\log^{\frac{2}{1-2c}}2N) \text{ and }M:=C\log^c(3+Q_0) + 2\pi H.
\end{equation}
This is possible since $c<\frac{1}{8}$, because this ensures $3\cdot c\cdot \frac{2}{1-2c}<1$.

Apply Proposition \ref{prop.datks} to get $1 \leq Q \leq Q_0$ and $\phi^*:\T_N \rightarrow \T_Q$ such that
\begin{equation}\label{eqn.efd}
\|\phi(x)-\phi^*(x)\|_\T \leq H/NQ \leq \eta \text{ for all }x \in \T_N.
\end{equation}
Let $\Gamma$ be the graph of $\phi^*$ \emph{i.e.\ }$\Gamma=\{(x,\phi^*(x)): x \in \T_N\} \subset \T_N \times \T_Q$, and for $s \in \Z_Q$ write $s=\tilde{s}+Q\Z$ for some $\tilde{s} \in \{1,\dots,Q\}$. Writing $f_s:=\exp(-2\pi i \wt{s}\phi)|_{\T_N}$, we have from the first inequality in (\ref{eqn.efd}) that
\begin{align*}
&\left|\wh{1_\Gamma}(r,s) -\frac{1}{Q}\wh{f_s}(r)\right|\\ & \qquad \qquad = \left|\frac{1}{NQ}\sum_{x \in \T_N}{\exp(-2\pi i xr)\exp(-2\pi i \wt{s}\phi(x))\left(\exp(-2\pi i \wt{s}(\phi^*(x)-\phi(x)))-1\right)}\right|\\ & \qquad \qquad \leq \frac{1}{NQ}\cdot N \cdot 2\pi \cdot \sup{\{\wt{s}: s \in \Z_Q\}}\cdot \frac{H}{NQ} = \frac{2\pi H}{NQ}.
\end{align*}
Hence
\begin{align*}
\|1_\Gamma\|_{A(\T_N \times \T_Q)} & \leq \sum_{s\in \Z_Q}{\frac{1}{Q}\|f_s\|_{A(\T_N)}} +2\pi H\\
& \leq \sup{\{\|\exp(-2\pi \wt{s}\phi)\|_{A(\T_N)}: s \in \Z_Q\}} +2\pi H \leq M.
\end{align*}
Apply Proposition \ref{prop.partialcohen} with $M$, $\eta$, the set $\Gamma$, and the group $\T_N \times \T_Q$ to get that there is $V \leq \T_N \times \T_Q$ such that $\|1_\Gamma \ast \mu_V\|_\infty \geq 1-\eta$ and
\begin{equation*}
|V| \geq \exp(-\refconstbig{partial}M^{3}\eta^{-3}\log^52M\eta^{-1})|\Gamma|>1
\end{equation*}
Let $x_0 \in \T_N \times \T_Q$ be such that $1_\Gamma \ast \mu_V(x_0)\geq 1-\eta$. Let $\pi:\T_N \times \T_Q \rightarrow \T_N$ be the canonical projection. Since $\Gamma$ is the graph of a function we know that $\pi$ is injective on $\Gamma$. Moreover $|\Gamma\cap (x_0+V)| = 1_\Gamma \ast \mu_V(x_0)|V|$, and $|V|\geq 2$ so
\begin{equation*}
|\pi(x_0+V)| \geq |\pi(\Gamma \cap (x_0+V))| \geq (1-\eta)|V| >1.
\end{equation*}
Now $\pi(x_0+V)$ is a coset of a subgroup of $\T_N$ and $N$ is prime, so $\pi(x_0+V)=\T_N$ and hence $|V|=|x_0+V| \geq N$. In the other direction $\frac{|\Gamma|}{|V|} \geq 1-\eta$ and hence $|V| \leq |\Gamma|/(1-\eta) < 2N$. Since $V$ has $\T_N$ as a quotient and $N$ is prime it follows by Lagrange's theorem that $|V|=N$.

Finally, write $W:=x_0+V$ so that
\begin{equation*}
|\Gamma \triangle W| =|\Gamma|+|W| - 2|\Gamma \cap (x_0+V)| \leq 2\eta |\Gamma|.
\end{equation*}
Between this and (\ref{eqn.efd}) the conditions of Lemma \ref{lem.limit} are established and this gives the result.
\end{proof}

The $\eta$-dependence in Proposition \ref{prop.partialcohen} is not important to us, but the $M$-dependence determines the quality of the bound in Theorem \ref{thm.main}. If $G$ is a cyclic group of prime order much larger than $\exp(M)$ and $A$ is an interval in $G$ with $\|1_A\|_{A(G)}=M$ then $|A| = \exp(\Omega(M))$, while we must have $|V| \leq 1$. It follows that in the lower bound the $M^3$-term cannot be improved to a smaller power than $M^1$. If a bound of that shape were known the remainder of the argument would give Theorem \ref{thm.main} with the $\frac{1}{8}$ replaced by $\frac{1}{4}$.

\section{Dirichlet's approximation theorem}

The elements of $\T=\R/\Z$ are cosets of $\Z$, and for $x \in \T$ we write $\|x\|_\T$ for $\min\{|t|: t \in x\}$.
\begin{theorem}[Dirichlet's approximation theorem]\label{thm.dat}
Suppose that $A$ is a set of size $d$, $\phi$ maps into $\T$ and has domain containing $A$, and $K \in \N^*$. Then there is an integer $1 \leq Q \leq K^d$ and $\phi^*:A \rightarrow \T_Q$ such that $\|\phi(x)-\phi^*(x)\|_\T \leq \frac{1}{KQ}$ for all $x \in A$.
\end{theorem}
Lebedev \cite[p124]{Lebedev:2012aa} used this in the case $A=\T_N$ and $K=N$, where $Q$ may be as large as $N^N$, and a key advance of Konyagin and Shkredov in \cite{konshk::0} was to show that in the setup of interest to us, one can do much better. 
\begin{proposition}\label{prop.datks}
Suppose that $C \geq 1$ and $c \in (0,1/2)$. Then there are constants $F(C,c), H(C,c)\geq 1 $ such that if $N \in \N^*$ and $\phi:\T \rightarrow \T$ is continuous with
\begin{equation}\label{eqn.bhhyp}
\|\exp(-2\pi i z \phi)\|_{A(\T_N)} \leq C\log^c(3+|z|) \text{ for all }z \in \Z.
\end{equation}
Then there is an integer $1 \leq Q \leq \exp(F(C,c)\log^{\frac{2}{1-2c}}2N)$ and $\phi^*:\T_N \rightarrow \T_Q$ such that $\|\phi(x)-\phi^*(x)\|_\T\leq \frac{H(C,c)}{NQ}$ for all $x \in \T_N$.
\end{proposition}
Our argument follows that of Konyagin and Shkredov. We introduce a little more smoothing to get the linear dependence on $Q$ in the final estimate above. This is not needed by Konyagin and Shkredov (who made do with square-root dependence in the corresponding estimate \cite[(46), p120]{konshk::0}), but smoothing in this way is completely routine.

We also make the aesthetic choice to work with a version of dissociativity rather than the direct argument in \cite[Lemma 4, p112]{konshk::0}. This is both slightly weaker as we cannot use a result like \cite[Lemma 9, p113]{konshk::0}, which could lead to savings of a power of a double-logarithm; and is not as generally presented as  \cite[Lemma 4, p112]{konshk::0} which is well-suited to use in other contexts.

For $G$ an Abelian group, $a \in G^d$ and $\sigma \in \Z^d$ write $\sigma\cdot a:=\sigma_1a_1+\cdots + \sigma_d a_d$. For $\phi:G \rightarrow H$, a map between groups, we write $\phi(a)$ for the element of $H^d$ defined by $\phi(a)_i:=\phi(a_i)$ for $1 \leq i \leq d$.
\begin{lemma}\label{lem.ks}
Suppose that $\phi:\T_N \rightarrow \T$ and $a \in \T_N^d$ is such that if $\sigma \in \{-1,0,1\}^d$ and $\sigma \cdot a=0$ and $\|\sigma \cdot \phi(a)\|_\T \leq \eta$, then $\sigma \equiv 0$. Then 
\begin{equation*}
d \leq \newconstbig{sksize}\cdot (\log 2N) \cdot \sup_{|z|\leq \frac{1}{2}\eta^{-1}}{\|\exp(-2\pi i z\phi)\|_{A(\T_N)}^2} .
\end{equation*}
\end{lemma}
\begin{proof}
Define
\begin{equation*}
F:\Z \rightarrow \C; z \mapsto \begin{cases}\eta & \text{ if }z =0\\ \eta \left(\frac{\sin \pi \eta z}{\pi \eta z}\right)^2 & \text{ otherwise.}\end{cases}
\end{equation*}
Then $F(z) \geq 0$ and we can compute that
\begin{equation}\label{eqn.invF}
\sum_{z \in \Z}{F(z) \exp(2\pi i zx)}=1_{\left[-\frac{1}{2}\eta,\frac{1}{2}\eta\right]} \ast \mu_{\left[-\frac{1}{2}\eta,\frac{1}{2}\eta\right]}(\|x\|_\T) \leq 1_{[-\eta,\eta]}(\|x\|_\T).
\end{equation}
Using $|\sin \theta| \geq \frac{2}{\pi}|\theta|$ whenever $|\theta| \leq \frac{\pi}{2}$, we also have
\begin{equation*}
\tau:=\sum_{|z| \leq \frac{1}{2}\eta^{-1}}{F(z)} \geq \sum_{|z| \leq \frac{1}{2}\eta^{-1}}{\eta \left(\frac{2}{\pi}\right)^2} \geq \frac{2}{\pi^2}.
\end{equation*}
For $a \in \T_N$ define 
\begin{equation*}
\lambda_a:\Z_N \times \Z \rightarrow S^1; (r,z) \mapsto  \exp(2\pi i (ar+\phi(a)z)),
\end{equation*} and define two probability measures $\wt{\mu}$ and $\mu$ on $\Z_N \times \Z$ with
\begin{equation*}
\wt{\mu}(\{r,z\}) = \frac{1}{N} F(z) \text{ and }\mu(\{(r,z)\})=\begin{cases}\frac{1}{\tau N} \cdot F(z) & \text{ if } |z| \leq \frac{1}{2}\eta^{-1}\\ 0 & \text{otherwise.}\end{cases}
\end{equation*}
Then, for $a,b \in \T_N$ we have the orthogonality relations
\begin{equation}\label{eqn.ortho}
\int{\lambda_a\overline{\lambda_b}\dd \mu} = \begin{cases}1 & \text{ if }a=b\\ 0 & \text{ otherwise.}\end{cases}
\end{equation}
From the hypothesis of the lemma, if $\sigma\cdot a=0$ and $\|\sigma\cdot \phi(a)\|_\T \leq \eta$ for $\sigma \in \{-1,0,1\}^d$, then $\sigma\equiv 0$, we conclude that $a_1,\dots,a_d$ are distinct; write $A:=\{a_1,\dots,a_d\}$.

Let $k:=\lceil \log 2N\rceil$ -- we make this choice now to be clear about dependencies, but it arises from optimising (\ref{eqn.optimise}) -- and define a linear operator
\begin{equation*}
T:\ell_2(A) \rightarrow L_{2k}(\mu); g \mapsto \sum_{a \in A}{g(a)\lambda_{a}}.
\end{equation*}
First we compute the norm of $T$.  Suppose $g \in \ell_2(A)$. By convexity of $\exp(ty)$ we have $\exp(ty) \leq \cosh t (1+ y \tanh t)$ whenever $t \in \R$ and $-1 \leq y \leq 1$, and so
\begin{align}
\nonumber  \int{\prod_{j=1}^d{\exp(\Re (g(a_j)\lambda_{a_j}))}\dd\mu}& \leq \tau^{-1}\int{\prod_{j=1}^d{\exp(\Re (g(a_j)\lambda_{a_j}))}\dd\wt{\mu}}\\ \nonumber & \leq \tau^{-1}\int{\prod_{j=1}^d{\cosh|g(a_j)|} \left(1+ \frac{\Re(g(a_j)\lambda_{a_j})}{|g(a_j)|}\tanh |g(a_j)|\right)\dd\wt{\mu}}\\
\label{eqn.de1}& =  \tau^{-1}\prod_{j=1}^d{\cosh |g(a_j)|}\int{\prod_{j=1}^d{(1+\Re \omega_j\lambda_{a_j})}\dd\wt{\mu}}
\end{align}
with the convention, and respectively notation, that 
\begin{equation*}
\frac{g(a_j)}{|g(a_j)|}\tanh |g(a_j)| =0 \text{ if }g(a_j)=0;\text{ resp.\ }\omega_j:=\frac{g(a_j)}{|g(a_j)|}\tanh |g(a_j)|.
\end{equation*}
Since $\Re z = \frac{1}{2}(z+\overline{z})$ we have
\begin{equation*}
\int{\prod_{j=1}^d{(1+\Re \omega_j\lambda_{a_j})}\dd\wt{\mu}} =\int{\sum_{ \sigma \in \{-1,0,1\}^d}{\frac{1}{2^{\|\sigma\|_{\ell_1}}} \cdot \prod_{j: \sigma_j=1}{\omega_j} \cdot \overline{\prod_{j: \sigma_j=-1}{\omega_j}}\cdot \lambda_{\sigma\cdot a} \dd \wt{\mu}}}.
\end{equation*}
However, by (\ref{eqn.invF}),
\begin{align*}
\int{\lambda_{\sigma \cdot a}\dd\wt{\mu}} & = \frac{1}{N}\sum_{r \in \Z_N}{\exp(2\pi i (\sigma\cdot a) r)} \cdot \sum_{z \in \Z}{F(z)\exp(2\pi i (\sigma \cdot \phi(a) )z)}\\  & = 1_{\{0\}}(\sigma \cdot a) 1_{\left[-\frac{1}{2}\eta,\frac{1}{2}\eta\right]} \ast \mu_{\left[-\frac{1}{2}\eta,\frac{1}{2}\eta\right]}(\|\sigma \cdot \phi(a)\|_\T) \leq 1_{\{0\}}(\sigma \cdot a) 1_{[-\eta,\eta]}(\|\sigma\cdot \phi(a)\|_\T).
\end{align*}
By the hypothesis of the lemma if $\sigma\cdot a =0$ and $\|\sigma \cdot \phi(a)\|_\T\leq \eta$ for $\sigma \in \{-1,0,1\}^d$ then $\sigma\equiv 0$, so we conclude that
\begin{equation*}
\int{\prod_{j=1}^d{(1+\Re \omega_j\lambda_{a_j})}\dd\wt{\mu}} \leq 1.
\end{equation*}
Combining this with (\ref{eqn.de1}) we get
\begin{equation*}
\int{\prod_{j=1}^d{\exp(\Re (g(a_j)\lambda_{a_j}))}\dd\mu}\leq \tau^{-1}\prod_{j=1}^d{\cosh|g(a_j)|} \leq \tau^{-1}\exp(\|g\|_{\ell_2(A)}^2/2),
\end{equation*}
with the last inequality since $\cosh x \leq \exp(x^2/2)$. The further inequality $\frac{1}{(2k)!}x^{2k} \leq \frac{1}{2}(\exp(x)+\exp(-x))$ in the preceding and (applied to $g$ and $-g$) gives
\begin{equation*}
\left\|\Re\sum_{a \in A}{g(a)\lambda_a}\right\|_{L_{2k}(\mu)}^{2k}\leq  (2k)!\cdot \tau^{-1}\exp(\|g\|_{\ell_2(A)}^2/2) \text{ for all }g \in \ell_2(A).
\end{equation*}
To compute the norm of $T$ we rescale $g$ to that $\|g\|_{\ell_2(A)}=2\sqrt{k}$, and then get
\begin{align*}
\frac{\|Tg\|_{L_{2k}(\mu)}}{\|g\|_{\ell_2(A)}} &\leq \frac{1}{2\sqrt{k}}\left(\|\Re Tg\|_{L_{2k}(\mu)} +\|\Re Tig\|_{L_{2k}(\mu)}\right)\\ &\leq \frac{1}{\sqrt{k}}\cdot (2k)!^{1/2k}\cdot \tau^{-1}\exp(4k/4k)=O(\sqrt{k}).
\end{align*}
It follows that $\|T\|=O(\sqrt{k})$. 

The adjoint of $T$ is the map
\begin{equation*}
T^*:L_{\frac{2k}{2k-1}}(\mu) \rightarrow \ell_2(A); f \mapsto \left(a\mapsto \int{f\overline{\lambda_a}\dd\mu}\right)
\end{equation*}
and $\|T^*\|=\|T\|=O(\sqrt{k})$. Now apply $T^*$ to $f:=\sum_{a \in \T_N}{\lambda_a}$ -- note that sum is over $\T_N$ \emph{not} $A$. From first the orthogonality relations (\ref{eqn.ortho}), and then $\log$-convexity of $L_p$-norms, we have
\begin{equation}\label{eqn.optimise}
|A| = \|T^*f\|_{\ell_2(A)}^2=O\left(k\|f\|_{L_{\frac{2k}{2k-1}}(\mu)}\right) = O\left(k\|f\|_{L_2(\mu)}^{\frac{2}{k}}\|f\|_{L_1(\mu)}^{\frac{2k-2}{k}}\right).
\end{equation}
Again from (\ref{eqn.ortho}) we have
\begin{equation}\label{eqn.2le}
\|f\|_{L_2(\mu)}^2 = \sum_{a,b \in \T_N}{\int{\lambda_a\overline{\lambda_b}\dd \mu} }=N,
\end{equation}
and, since $|f|=|\sum_{a \in A}{\overline{\lambda_a}}|$,
\begin{align}
\nonumber \|f\|_{L_1(\mu)} & = \tau^{-1}\sum_{|z| \leq \frac{1}{2}\eta^{-1}}{F(z)\sum_{r \in \Z_N}{\left|\frac{1}{N}\sum_{a \in \T_N}{\exp(-2\pi i ar)\exp(-2\pi i z\phi(a))}\right|}}\\ \nonumber & = O\left(\sum_{|z| \leq \frac{1}{2}\eta^{-1}}{F(z)\|\exp(-2\pi i z\phi)\|_{A(\T_N)}}\right)\\ & \label{eqn.l1e} =O\left(\sup_{|z| \leq \frac{1}{2}\eta^{-1}}{\|\exp(-2\pi i z\phi)\|_{A(\T_N)}}\right).
\end{align}
Putting (\ref{eqn.2le}) and (\ref{eqn.l1e}) into (\ref{eqn.optimise}) and recalling that $k=\lceil \log 2N\rceil$ gives the result.
\end{proof}
\begin{proof}[Proof of Proposition \ref{prop.datks}]
Set
\begin{equation}\label{eqn.cont}
d_*:=\lfloor(\refconstbig{sksize}\cdot (2C)^{2c})^\frac{1}{1-2c}\log^{\frac{1+2c}{1-2c}}2N\rfloor \text{ and } \eta:=1/N^{2d_*+2}.
\end{equation}
Suppose that $a \in \T_N^{d}$ is such that if $\sigma \in \{-1,0,1\}^{d}$, $\sigma \cdot a =0$, and $\|\sigma \cdot \phi(a)\|_\T\leq \eta$, then $\sigma \equiv 0$. If $d>d_*$ then we may reduce $d$ to be equal to $d_*+1$ by throwing out $a_{d_*+2},\dots,a_{d}$. Hence we shall assume, for a contradiction that $d=d_*+1$. 

By Lemma \ref{lem.ks} applied to $a$, and using (\ref{eqn.bhhyp}), we have
\begin{equation*}
d_*+1 \leq \refconstbig{sksize}\cdot (\log 2N) \cdot C^{2c}\log^{2c}\left(3+\frac{1}{2}N^{2d_*+2}\right) \leq \refconstbig{sksize}\cdot C^{2c}\cdot (\log 2N)^{1+2c} \cdot (2d_*+2)^{2c}.
\end{equation*}
This rearranges to give
\begin{equation*}
 d_*+1 \leq \left(\refconstbig{sksize}\cdot (2C)^{2c}\cdot \log^{1+2c} 2N\right)^{\frac{1}{1-2c}},
 \end{equation*}
 which contradicts (\ref{eqn.cont}). Let $d$ be maximal such that there is $a \in \T_N^{d}$ such that if $\sigma \in \{-1,0,1\}^{d}$, $\sigma \cdot a =0$, and $\|\sigma \cdot \phi(a)\|_\T\leq \eta$, then $\sigma \equiv 0$. It follows from the previous that such a $d$ exists and in fact $d \leq d_*$; fix an $a \in \T_N^d$ that bears witness to this.
 
 Let $A=\{a_1,\dots,a_{d}\}$ and apply Dirichlet's approximation theorem (Theorem \ref{thm.dat} with $K=N^2$) to get some $1 \leq Q \leq N^{2d}$ and $\wt{\phi}:A \rightarrow \T_Q$ such that
\begin{equation}\label{eqn.phitilde}
\|\wt{\phi}(a_i) - \phi(a_i)\|_\T \leq \frac{1}{N^2Q} \text{ for all }1 \leq i \leq d.
\end{equation}
If $x \in \T_N$ then writing $a':=(a_1,\dots,a_d,x) \in\T_N^{d+1}$ we know by maximality of $d$ that there is $\sigma' \in \{-1,0,1\}^{d+1}$, $\sigma' \cdot a' = 0$ and $\|\sigma' \cdot \phi(a')\|_\T \leq \eta$. In view of the choice of $A$ we also know that $\sigma'_{d+1} \neq 0$, and by multiplying all the entries of $\sigma'$ through by $-1$ if necessary we may assume that $\sigma'_{d+1}=-1$. Writing $\sigma=(\sigma_1',\dots,\sigma_d')$ we have
\begin{equation}\label{eqn.rep}
x=\sigma \cdot a \text{ and } \|\phi(x) - \sigma\cdot \phi(a)\|_\T \leq \frac{1}{N^2Q}.
\end{equation}
For each $x \in \T_N$ pick a $\sigma^x \in \{-1,0,1\}^d$ such that (\ref{eqn.rep}) holds. Define
\begin{equation*}
\phi^*:\T_N \rightarrow \T_Q; x \mapsto \sigma^x \cdot \wt{\phi}(a).
\end{equation*}
Then from (\ref{eqn.phitilde}) and (\ref{eqn.rep}) we have
\begin{equation*}
\|\phi^*(x)-\phi(x)\|_\T \leq \sum_{\sigma^x_i \neq 0}{\|\wt{\phi}(a_i)-\phi(a_i)\|_\T} + \|\sigma^x\cdot \phi(a)-\phi(x)\|_\T \leq \frac{d+1}{N^2Q}=O_{C,c}(1/NQ),
\end{equation*}
since $d \leq d_*$. Since $Q\leq N^{2d}\leq N^{2d_*}$ and the result is proved.
\end{proof}

\section{Cohen's idempotent theorem}

Cohen's idempotent theorem from \cite{coh::} describes the structure of $f \in B(G)$ taking the values $0$ or $1$.  It is discussed in detail in \cite[Chapter 3]{rud::1}. In the case of finite Abelian groups it is contentless but there is a quantitative strengthening of Green and the author \cite{gresan::0}, and we shall need the ingredients of the proof of this.

It is fairly natural to connect Cohen's idempotent theorem with the Beurling-Helson Theorem: the former is the key ingredient in the description in \cite{Cohen:1960aa} of the homomorphisms $L_1(G) \rightarrow M(H)$ where $G$ and $H$ are locally compact Abelian groups, and the latter gives this description in a special case. See \cite[Theorem 4.1.3, p78]{rud::1} and the discussion following. In fact Cohen's theorem was already applied to prove a generalisation of the Beurling-Helson theorem in \cite{Domar:1973aa}.

The main result of this section is the following:
\begin{proposition}\label{prop.partialcohen}
Suppose that $G$ is a finite Abelian group, $1_A \in A(G)$ has $\|1_A\|_{A(G)} \leq M$ for some $M\geq 1$, and $\eta \in (0,1/4]$. Then there is $V \leq G$ with $\|1_A \ast \mu_V\|_{L_\infty(G)} \geq 1-\eta$ and $|V| \geq \exp(-\newconstbig{partial}M^{3}\eta^{-3}\log^52M\eta^{-1})|A|$.
\end{proposition}

For the remainder of this section $G$ denotes a finite Abelian group and $\wh{G}$ the group of homomorphisms $G \rightarrow S^1$ under pointwise multiplication. Following custom \cite[p7]{rud::1} we sometimes write $\langle x,\gamma\rangle$ for $\gamma(x)$, and then the Fourier transform is defined as in (\ref{eqn.ft}). 

For $f:G \rightarrow \C$ and $x \in G$ we write
\begin{equation*}
\tau_x(f)(y):=f(y-x) \text{ for all }y \in G.
\end{equation*}
For $\mu $ a measure on $G$ (with $\sigma$-algebra $\mathcal{P}(G)$) we write $\wt{\mu}$ for the measure assigning mass $\overline{\mu}(-E)$ to $E \subset G$, and if $f,g:G \rightarrow \C$ too then
\begin{equation*}
f \ast \mu(x)=\int{f(y)\dd\mu(x-y)} \text{ and } f \ast g(x):=\int{f(y)g(x-y)\dd\mu_G(y)} \text{ for all }x \in G.
\end{equation*}

We shall use ingredients from \cite{san::12} for proving a quantitative version of Cohen's idempotent theorem. That paper has some nonstandard terminology which was introduced to make the proofs there easier, and we record the relevant parts now.

Given $\Lambda \subset \wh{G}$ and a function $\delta:\Lambda \rightarrow \R_{>0}$ we write
\begin{equation*}
\Bohr(\Lambda,\delta):=\{x \in G: \|\lambda(x)\|_\T<\delta(\lambda) \text{ for all }\lambda \in \Lambda\}.
\end{equation*}
This is a slight, but completely natural, generalisation of the usual definition of Bohr set from \emph{e.g.\ }\cite[Definition 4.1, p187]{taovu::} which takes $\delta$ to be a constant function.

A Bohr system is a vector $B=(B_\eta)_{\eta \in (0,1]}$ where $B_\eta=\Bohr(\Lambda,\eta\delta)$ for each $\eta \in (0,1]$. Bohr systems are determined by the set $\Lambda$ and the function $\delta$, though different sets and functions may determine the same Bohr system. 

For $S,T \subset G$, the covering number of $S$ by $T$ is
\begin{equation}\label{eqn.covin}
\mathcal{C}(S;T):=\min\{|X|: S \subset X+T\} \geq \frac{|S|}{|T|}.
\end{equation}
In view of the inequality we think of this as measuring the size of $S$ relative to $T$.

Covering numbers are used to define the doubling dimension of Bohr systems:
\begin{equation}\label{eqn.dd}
\dim^* B = \sup\left\{\log_2\mathcal{C}\left(B_\eta;B_{\frac{1}{2}\eta}\right): \eta \in (0,1]\right\},
\end{equation}
which capture for us how Bohr systems grow. One may think of the classic size bounds on Bohr sets in \cite[Lemma 4.19]{taovu::} as saying that a Bohr system where $\Lambda$ has size $k$ and $\delta$ is a constant function has doubling dimension $O(k)$.

For the results we need there are two other definitions. The first is that of difference covering number \cite[p6]{san::12} denoted $\mathcal{C}^\Delta(S;T)\label{eqn.dcn}$. It turns out this is equivalent to the covering numbers of some related sets (see \cite[Lemma 2.4 (iii) \& (iv), p6]{san::12}), but for our work here it just matters that it is the same term in each of Propositions \ref{prop.cts} \& \ref{prop.frei}. The second is the dimension of a Bohr system \cite[p13]{san::12}. Here all that matters to us is that it is equivalent to the doubling dimension of the Bohr system up to multiplicative constants (see \cite[Lemma 3.7 (iii), p13]{san::12}).
 
The reason for these two definitions is that they behave better w.r.t.\ intersections of Bohr systems than respectively covering numbers (see \cite[Lemma 2.4 (ii), p6]{san::12}), and doubling dimension (see \cite[Lemma 3.7 (i), p13]{san::12}). This makes the proof of Proposition \ref{prop.cts} in \cite{san::12} cleaner. We do not need this better behaviour for the present work but we include them so that the reader can quickly appreciate the differences between the statements of Propositions \ref{prop.cts} \& \ref{prop.frei} below and those in their original source.

\begin{proposition}[{\cite[Proposition 7.1, p25]{san::12}}]\label{prop.cts}
Suppose that $G$ is a finite Abelian group, $B$ is a Bohr system with $\dim^*B \leq d$ (for some $d \geq 1$), $f \in A(G)$, and $\delta,\kappa \in (0,1]$ and $p \geq 1$ are parameters.  Then there is a Bohr system $B'$ with $B_1'\subset B_1$, and probability measures $\mu$ and $\nu$ with $\supp \nu \subset B_\kappa'$ such that 
\begin{enumerate}[label=(\roman*)]
\item \label{pt.cts}
\begin{equation*}
\dim^* B' \leq 2d + O(p\delta^{-2}\log^22\delta^{-1});
\end{equation*}
and for any $S \subset G$ we have
\begin{equation*}
\tikzmark{b3}\mathcal{C}(S;B_1') \leq \exp(O(\delta^{-1}d\log 2\kappa^{-1}d + p\delta^{-3}\log^32p\kappa^{-1}\delta^{-1})) \mathcal{C}^{\Delta}\left(S;B_{1}\right)\tikzmark{b2}
\end{equation*}
\begin{tikzpicture}[overlay, remember picture, >=stealth]

\node[font=\tiny, right=2cm of {pic cs:b2},  align=left, red] (explain2)
{difference\\ covering\\ number\\  discussed\\ on p\pageref{eqn.dcn};\\ it is a variant\\ of covering\\ number (\ref{eqn.covin});\\  it fits with\\ the output\\ of Prop.\ \ref{prop.frei}};

\draw[<-, red,rounded corners=0pt]
  ([yshift=4pt]pic cs:b2) -- ([yshift=4pt]explain2.west) ;
\end{tikzpicture}
\begin{tikzpicture}[overlay, remember picture, >=stealth]

\node[font=\tiny, left=2cm of {pic cs:b3},  align=left, red] (explain1)
{covering\\ number (\ref{eqn.covin})};

\draw[->, red]
  ([yshift=4pt]explain1.east) -- ([yshift=4pt, xshift=-3pt]pic cs:b3);
\end{tikzpicture}
\item\label{pt.2} \footnote{This part of the conclusion is stated in terms of $B'$-approximately invariant probability measures in \cite[Proposition 7.1]{san::12}. The definition of these is on \cite[p16]{san::12}, and here we have recorded the consequence of this definition from \cite[Lemma 4.2, p17]{san::12}.}
\begin{equation*}
\|\tau_x(f\ast \mu) - f \ast \mu\|_{L_\infty(G)} \leq \eta \|f\|_{L_\infty(G)}\text{ for all }x \in B_{\frac{1}{2}\eta}' \text{ and all }\eta \in (0,1];
\end{equation*}
\item \label{pt.3}and
\begin{equation*}
\sup_{x \in G}{\|f - f\ast \mu\|_{L_p(\tau_x(\nu))}} \leq \delta \|f\|_{A(G)}.
\end{equation*}
\end{enumerate}
\end{proposition}
The intuition here is that we are given a Bohr set $B_1$, and a function $f\in A(G)$ which we expect to be continuous in a sense we make quantitative. (Without this quantitative aspect the continuity is contentless when $G$ is finite.) The proposition finds two smaller Bohr sets $B_1'$ and $B_{\kappa}'$ such that $B_1' +B_{\kappa}' \approx B_1'$ and having properties \ref{pt.cts}--\ref{pt.3}.

Property \ref{pt.cts} tells us that $B_1'$ is not too small, and that for any set $S$ if $S$ has large intersection with a translate of $B_1$ then it also has large intersection with a translate of $B_1'$. In Properties \ref{pt.2} \& \ref{pt.3} it is helpful to pretend that $\mu$ is $\mu_{B_1}$, $\nu$ is $\mu_{B_{\kappa}'}$, and $\eta \approx \kappa$ so that \ref{pt.2} captures $B_1' +B_{\kappa}' \approx B_1'$, and \ref{pt.3} tells us that $f$ is very close to constant on translates of $B_{\kappa}'$ -- this is the sense in which the continuity is made quantitative.

\begin{proposition}[{\cite[Proposition 8.1, p32]{san::12}}]\label{prop.frei}
Suppose that $G$ is a finite Abelian group and $S$ is non-empty with $\mu_G(S+S)\leq K\mu_G(S)$.  Then there is a Bohr system $B$ with
\Needspace{6\baselineskip}
\begin{equation*}
\tikzmark{b1} \mathcal{C}^{\Delta}(S;B_1) = \exp(O(\log^32K (\log \log 3K)^4))
\end{equation*}
\begin{tikzpicture}[overlay, remember picture, >=stealth]

\node[font=\tiny, left=4cm of {pic cs:b1},  align=left, red] (explain)
{difference\\ covering\\ number\\  discussed\\ on p\pageref{eqn.dcn};\\ it is a variant\\ of covering\\ number (\ref{eqn.covin});\\  it fits with\\ the input\\ of Prop.\ \ref{prop.cts}};

\draw[->, red]
  ([yshift=4pt]explain.east) -- ([yshift=4pt, xshift=-3pt]pic cs:b1);
\end{tikzpicture}
and
\begin{equation*}
\dim^* B=O(\log^32K (\log \log 3K)^4),
\end{equation*}
such that
\begin{equation}\label{eqn.jp}
\|1_S \ast \beta\|_{L_\infty(G)} =\exp(-O(\log2K (\log \log 3K)))
\end{equation}
for any probability measure $\beta$ supported on $B_1$.
\end{proposition}
This proposition is a routine Freiman-type theorem with the Bohr system playing the role of the coset-progression that usually appears.

\begin{proof}[Proof of Proposition \ref{prop.partialcohen}] The Fourier transform takes convolution to multiplication and so by Parseval's theorem and $\log$-convexity of $\ell_p$-norms we have
\begin{equation*}
\int{(1_A \ast 1_{-A})^2\dd \mu_G} = \|\wh{1_A}\|_{\ell_4(\Gamma)}^4 \geq \frac{\|\wh{1_A}\|_{\ell_2(\Gamma)}^6}{\|\wh{1_A}\|_{\ell_1(\Gamma)}^2}\geq \frac{\mu_G(A)^3}{M^2}.
\end{equation*}
Since $G$ is finite, the integral on the left is $|G|^{-3}$ times the additive energy of $A$, that is times the number of quadruples $(x,y,z,w) \in A^4$ such that $x+y=z+w$.

By the Balog-Szemer{\'e}di-Gowers Theorem\footnote{We want it in the form that says if the additive energy of $A$ (that is $E(A,A)$ from \cite[Definition 2.8, p78]{taovu::}) is at least $\kappa |A|^3$ then there is a subset $S \subset A$ with $|S| \geq \kappa^{O(1)}|A|$ and $|S+S| \leq \kappa^{-O(1)}|S|$. This can be found in \cite{taovu::} by combining Theorem 2.31 (i) {$\Rightarrow$} (iv), p100, with Proposition 2.27 (i) {$\Rightarrow$} (vi), p93, recalling the definition of $\sigma[A]=\frac{|A+A|}{|A|}$ from Definition 2.4, p74.} there is $S \subset A$ with $|S| \geq M^{-O(1)}|A|$ and $\mu_G(S+S) \leq M^{O(1)}\mu_G(S)$. By Proposition \ref{prop.frei} applied to $S$ there is a Bohr system $B$ with
\begin{equation}\label{eqn.ku}
\mathcal{C}^\Delta(S;B_1) \leq \exp(O(\log^{3}2M(\log \log 3M)^4)),
\end{equation}
\begin{equation*}
d:=\dim^*B = O(\log^{3}2M(\log \log 3M)^4),
\end{equation*}
and
\begin{equation}\label{eqn.low}
\|1_{S} \ast \beta \|_{L_\infty(G)} \geq \alpha_M:=(2M)^{-O(\log \log 3M)}
\end{equation}
for any probability measure $\beta$ supported on $B_1$.

We then apply Proposition \ref{prop.cts} with $f=1_A$, $\delta:=\frac{\eta}{4M}$, $p:= \log_{\frac{3}{2}} 2\alpha_M^{-1}$, $\kappa:=\frac{1}{8}\eta$, to get a Bohr system $B'$ with $B'_1 \subset B_1$, and probability measures $\mu$ and $\nu$ with $\supp \nu \subset B_\kappa'$ such that
\begin{equation}\label{eqn.b*}
\dim^*B' =O(d + p\delta^{-3}\log^22\delta^{-1})=O(M^{3}\eta^{-3}\log^42M\eta^{-1})
\end{equation}
and
\begin{align}
\nonumber \frac{|S|}{|B_1'|} \leq \mathcal{C}(S;B_1')& \leq \exp(O(\delta^{-1}d\log 2\kappa^{-1}d + p\delta^{-3}\log^32p\kappa^{-1}\delta^{-1}))\mathcal{C}^\Delta\left(S;B_1\right)\\ \label{eqn.size}  & = \exp(O(M^{3}\eta^{-3}\log^52M\eta^{-1}))\mathcal{C}^\Delta\left(S;B_1\right);
\end{align}
and such that
\begin{equation}\label{eqn.const}
|\tau_y(1_A \ast \mu)(x) -1_A \ast \mu(x)| \leq \frac{1}{4}\eta \text{ for all } y \in B_\kappa' \text{ and }x\in G;
\end{equation}
and
\begin{equation}\label{eqn.invA}
\|1_A -1_A \ast \mu\|_{L_p(\tau_x(\nu))} \leq \frac{1}{4}\eta \text{ for all }x \in G.
\end{equation}
Since $\nu$ is supported in $B_\kappa'$, and $B_\kappa'$ is symmetric, $\wt{\nu}$ is a probability measure supported in $B_{\kappa}'$ too. It follows that
\begin{equation*}
\int{|1_A \ast \mu(z)-1_A \ast \mu(x)|^p\dd \tau_x(\nu)(z)}=\int{|\tau_y(1_A \ast \mu)(x) -1_A \ast\mu(x)|^p\dd \wt{\nu}(y)} \leq \left(\frac{1}{4}\eta\right)^p
\end{equation*}
for all $x \in G$, where the inequality is from (\ref{eqn.const}). Hence
\begin{equation*}
\|1_A \ast \mu - 1_A \ast \mu(x)\|_{L_p(\tau_x(\nu))} \leq \frac{1}{4}\eta \text{ for all }x \in G.
\end{equation*}
Combined with (\ref{eqn.invA}) this implies that
\begin{equation*}
\|1_A -1_A \ast \mu(x)\|_{L_p(\tau_x(\nu))}\leq \frac{1}{2}\eta \text{ for all } x \in G.
\end{equation*}
In particular
\begin{equation}\label{eqn.specific}
1_A \ast \wt{\nu}(x)|1-1_A \ast \mu(x)|^p + (1-1_A \ast \wt{\nu}(x))|1_A \ast \mu(x)|^p \leq \frac{\eta^p}{2^p} \text{ for all }x \in G.
\end{equation}
Since $1_A \ast \wt{\nu}(x) \in [0,1]$, either $1_A \ast \wt{\nu}(x)\geq \frac{1}{2}$ or $1-1_A \ast \wt{\nu}(x) \geq \frac{1}{2}$, and hence either
\begin{equation}\label{eqn.ai}
|1-1_A \ast \mu(x)| \leq \eta \text{ or }|1_A \ast \mu(x)| \leq \eta \text{ for all }x \in G,
\end{equation}
where the particular inequality that holds may depend on $x$.

Define
\begin{equation*}
T:=\{x\in G: |1-1_A \ast \mu(x)| \leq \eta\}.
\end{equation*}
If $x \in T$ then by (\ref{eqn.const}) $|1-1_A \ast \mu(x-y)| \leq \frac{5}{4}\eta$ for all $y \in B'_\kappa$. Since $\eta<\frac{4}{9}$ we have by (\ref{eqn.ai}) that $|1-1_A \ast \mu(x-y)| \leq \eta$ \emph{i.e.\ }$x-y \in B_\kappa$. In particular, since $0 \in B_\kappa$ we have $T-B_\kappa=T$, and so $T+V=T$ for $V$ the subgroup generated by $B_\kappa$.

Since $B_\kappa' \subset B_1$, $S \subset A$ and $\supp \wt{\nu} \subset -B_\kappa'=B_\kappa'$ we get from (\ref{eqn.low}) that there is some $x_0\in G$ such that $1_{S} \ast \wt{\nu}(x_0)=\|1_{S} \ast \wt{\nu}\|_{L_\infty(G)}$, whence
\begin{equation*}
1_A \ast \wt{\nu}(x_0)=\int{1_A\dd \tau_{x_0}(\nu)} \geq \int{1_{S}\dd \tau_{x_0}(\nu)} =\|1_{S} \ast \wt{\nu}\|_{L_\infty(G)} \geq \alpha_M.
\end{equation*}
If $|1_A \ast \mu(x_0)| \leq \eta$, then since $\eta \leq \frac{1}{4}$ we have from (\ref{eqn.specific}) that
\begin{equation*}
\alpha_M\left(\frac{3}{4}\right)^p  \leq 1_A \ast \wt{\nu}(x_0)(1-\eta)^p \leq 2^{-p}
\end{equation*}
which contradicts the choice of $p$, so by (\ref{eqn.ai}) we have $x_0\in T$.

From the definition of doubling dimension (\ref{eqn.dd}) and the inequality in (\ref{eqn.covin}) for covering numbers we have
\begin{equation*}
|B_\kappa'| \geq 2^{-\lceil \log_2\kappa^{-1}\rceil\dim^* B'}|B_1'|.
\end{equation*}
Inserting the upper bound on $\dim^*B'$ from (\ref{eqn.b*}), the lower bound on $|B'_1|$ in (\ref{eqn.size}), the upper bound on $\mathcal{C}^\Delta(S;B_1)$ from (\ref{eqn.ku}) and finally the fact that $|S| \geq M^{-O(1)}|A|$ we get
\begin{equation*}
|B_\kappa'| \geq \exp(-O(M^{3}\eta^{-3}\log^52M\eta^{-1}))|A|.
\end{equation*}
Now $|V| \geq |B_\kappa'|$ and also 
\begin{equation*}
\int{1_A \ast \mu \dd \tau_{x_0}\mu_V} \geq 1-\eta
\end{equation*}
since $x_0 +V \subset T$, and hence $\|1_A \ast \mu_V\|_{L_\infty(G)} \geq 1-\eta$. The result is proved.
\end{proof}

\section{The limiting argument}

In this final section we record a version of the limiting argument relevant to our circumstances.
\begin{lemma}\label{lem.limit}
Suppose that $\phi:\T \rightarrow \T$ is continuous and for all $\eta \in (0,1]$ and $N_0\in \N^*$ there is an integer $N \geq N_0$, $\phi^*:\T_N \rightarrow \T$ with graph $\Gamma:=\{(x,\phi^*(x)): x \in\T_N\}$, and $W$ such that
\begin{enumerate}[label=(\roman*)]
\item  $\|\phi(x)-\phi^*(x)\|_\T \leq \eta$ for all $x \in \T_N$;
\item  $W$ is a coset of a finite subgroup of $\T^2$ such that $|\Gamma \triangle W| \leq \eta N$.
\end{enumerate}
Then there is $t\in \T$ and $w \in \Z$ such that $\phi(x)=wx+t$ for all $x \in \T$.
\end{lemma}
\begin{proof}
Define $\Phi$ on $\T^3$ by
\begin{equation*}
\Phi(x,y,z):=\exp(2\pi i (\phi(x+y+z)-\phi(x+y)-\phi(x+z) +\phi(x))).
\end{equation*}
If $\Phi(x,y,z)=1$ for all $x,y,z \in \T$ then the map $x \mapsto \phi(x)-\phi(0)$ is a homomorphism. Since it is also continuous, it has the form $x \mapsto wx$ for some $w \in \Z$ (see \cite[p13]{rud::1}), and the result is proved.

It follows that we may assume that $\Phi(x,y,z) \neq 0$ for some $x,y,z \in \T$ and hence, since $\Phi$ is continuous, that
\begin{equation}\label{eqn.kappa}
\kappa:=\int{|\Phi-1|\dd \mu_{\T^3}}>0.
\end{equation}
Again, by continuity of $\Phi$ and the fact that $\mu_{\T_N}$ converges weakly to $\mu_\T$ as $N \rightarrow \infty$, there is $N_0$ such that for all $N\geq N_0$ we have
\begin{equation}\label{eqn.contra}
\left|\int{|\Phi-1|\dd \mu_{\T^3}} -\int{|\Phi-1|\dd \mu_{\T_N^3}} \right|<\frac{1}{2}\kappa.
\end{equation}

Now, by the hypotheses of the lemma applied with $\eta:=\frac{1}{4}\min\{1,\frac{1}{2\refconstbig{5}},\frac{1}{\refconstbig{exp}}\}\kappa>0$ there is $N \geq N_0$ and $\phi^*:\T_N \rightarrow \T$ with
\begin{equation*}
\|\phi(x)-\phi^*(x)\|_\T\leq \eta \text{ for all } x \in \T_N.
\end{equation*}
Define $\Phi^*$ on $\T_N^3$ by
\begin{equation*}
\Phi^*(x,y,z):=\exp(2\pi i (\phi^*(x+y+z)-\phi^*(x+y)-\phi^*(x+z) +\phi^*(x))),
\end{equation*}
so that
\begin{equation*}
|\Phi^*(x,y,z) - \Phi(x,y,z)| \leq \newconstbig{exp}\max\{\|\phi^*(t)-\phi(t)\|_\T: t \in \{x,x+y,x+z,x+y+z\}\},
\end{equation*}
and hence 
\begin{equation}\label{eqn.expest}
|\Phi^*(x,y,z) - \Phi(x,y,z)| \leq \refconstbig{exp}\eta  \text{ for all }x,y,z \in \T_N.
\end{equation}
The hypotheses also give $W$, a coset of a subgroup of $\T^2$, such that if $\Gamma$ is the graph of $\phi^*$ then $|\Gamma\triangle W|\leq 2\eta N$.

There is a one-to-one correspondence between triples $(x,y,z)\in \T_N^3$ and triples $(\alpha,\beta,\gamma)$ with $\alpha,\alpha+\beta,\alpha+\gamma \in \Gamma$ given by $\alpha:=(x,\phi^*(x)) \in \Gamma$, $\alpha+\beta:= (x+y,\phi^*(x+y))\in \Gamma$ and $\alpha+\gamma:=(x+z,\phi^*(x+z))\in \Gamma$. Moreover, $\Phi^*(x,y,z)\neq 1$ if and only if $\alpha+\beta+\gamma \notin \Gamma$. Hence there are at most
\begin{align*}
&\sum_{\alpha,\beta,\gamma}{1_{\Gamma^c}(\alpha+\beta+\gamma)1_\Gamma(\alpha+\beta)1_\Gamma(\alpha+\gamma)1_\Gamma(\alpha)}\\
& \qquad \qquad = |\Gamma|^3-\sum_{\alpha,\beta,\gamma}{1_\Gamma(\alpha+\beta+\gamma)1_\Gamma(\alpha+\beta)1_\Gamma(\alpha+\gamma)1_\Gamma(\alpha)}
\end{align*}
triples $(x,y,z)\in \T_N^3$ with $\Phi^*(x,y,z)\neq 1$. Now
\begin{align*}
& \sum_{\alpha,\beta,\gamma}{1_\Gamma(\alpha+\beta+\gamma)1_\Gamma(\alpha+\beta)1_\Gamma(\alpha+\gamma)1_\Gamma(\alpha)}\\ & \qquad \qquad \geq  \sum_{\alpha,\beta,\gamma}{1_{W\cap \Gamma}(\alpha+\beta+\gamma)1_{W\cap \Gamma}(\alpha+\beta)1_{W\cap \Gamma}(\alpha+\gamma)1_{W\cap \Gamma}(\alpha)}\\
&\qquad \qquad =  \sum_{\alpha}{\left(\sum_{\beta}{1_{W\cap \Gamma}(\alpha+\beta)1_{W \cap \Gamma}(\beta)}\right)^2}\\
&\qquad \qquad \geq \frac{1}{|W\cap \Gamma - W\cap \Gamma|}\left(\sum_{\alpha,\beta}{1_{W\cap \Gamma}(\alpha+\beta)1_{W \cap \Gamma}(\beta)}\right)^2\\ &\qquad \qquad  =\frac{|W \cap \Gamma|^4}{|W\cap \Gamma - W\cap \Gamma|}\geq \frac{|W\cap \Gamma|^4}{|W|}\geq \frac{(1-2\eta)^4}{1+2\eta}|\Gamma|^3 \geq (1-\newconstbig{5}\eta)|\Gamma|^3.
\end{align*}
It follows that the number of triples $(x,y,z) \in \T_N^3$ such that $\Phi^*(x,y,z)\neq 1$ is at most $\refconstbig{5} \eta N^3$. Hence, by (\ref{eqn.expest}),\begin{equation*}
\int{|\Phi-1|\dd \mu_{\T_N^3}} \leq \int{|\Phi^*-1|\dd \mu_{\T_N^3} }+ \int{|\Phi^*-\Phi|\dd \mu_{\T_N^3}} \leq 2\refconstbig{5}\eta +\refconstbig{exp}\eta\leq\frac{1}{2}\kappa,
\end{equation*}
which is a contradiction to (\ref{eqn.contra}) and (\ref{eqn.kappa}).
\end{proof}

\bibliographystyle{halpha}

\bibliography{references}

\end{document}